\documentclass[12pt,reqno]{amsart}

\usepackage{amsmath,amsfonts,amssymb,amscd,amsthm,calc,enumerate}

\newtheorem{theorem}{Theorem}[section]

\newtheorem{proposition}[theorem]{Proposition}

\theoremstyle{definition}
\newtheorem{definition}[theorem]{\bf Definition}
\newcommand{\darrow}{\!\downarrow}
\newcommand{\uarrow}{\!\uparrow}
\newcommand{\la}{\langle}
\newcommand{\ra}{\rangle}

\newcommand{\bigset}[1]{\big\{ #1 \big\}}
\newcommand{\restr}{\mbox{\raisebox{.5mm}{$\upharpoonright$}}}

\renewcommand{\leq}{\leqslant}

\newcommand{\vph}{\varphi}

\newcommand{\dom}{\mathrm{dom}}
\newcommand{\FPF}{\mathrm{FPF}}
\newcommand{\DNC}{\mathrm{DNC}}

\begin{document}

\title[Generalizations of the recursion theorem]
{Generalizations of the recursion theorem}

\author[S. A. Terwijn]{Sebastiaan A. Terwijn}
\address[Sebastiaan A. Terwijn]{Radboud University Nijmegen\\
Department of Mathematics\\
P.O. Box 9010, 6500 GL Nijmegen, the Netherlands.
} \email{terwijn@math.ru.nl}

\begin{abstract}
We consider two generalizations of the recursion theorem, namely 
Visser's ADN theorem and Arslanov's completeness criterion, 
and we prove a joint generalization of these theorems. 
\end{abstract}



\date{\today}

\maketitle

\section{Introduction}

The recursion theorem is a classic result in computability theory. 
It was found by S. C. Kleene in his study of the $\lambda$-calculus, 
and first appeared in~\cite{Kleene}. (It appears somewhat hidden, 
on p153, in the last paragraph of section~2.)  
In the following, $\vph_n$ denotes the $n$-th partial computable 
(p.c.) function, in some standard numbering of the p.c.\ functions. 

\begin{theorem} {\rm (The recursion theorem, Kleene~\cite{Kleene})}\label{recthm}
Let $f$ be a computable function. Then $f$ has a fixed point, 
i.e.\ there exists a number $e$ such that $\vph_{f(e)} = \vph_e$. 
\end{theorem}

The recursion theorem allows for certain kinds of circular definitions,
which gives the result an air of mystery.
The proof (translated from the $\lambda$-calculus) of this theorem is 
very short, but somewhat enigmatic. An illuminating way to view it, 
namely as a diagonalization argument that fails, was given by 
Owings~\cite{Owings}.

Kleene actually proved the following version of the recursion 
theorem, that we will use below. For a discussion of this second 
version see Moschovakis~\cite{Moschovakis}. 

\begin{theorem} {\rm (The recursion theorem with parameters, 
Kleene~\cite{Kleene})} \label{recthmparam}
Let $h(n,x)$ be a computable binary function. 
Then there exists a computable function $f$ such that for all~$n$, 
$\vph_{f(n)} = \vph_{h(n,f(n))}$.
\end{theorem}

Note that for every fixed $n$, the computable function $h(n,x)$ has 
a fixed point by the recursion theorem. The second version with parameters 
says that this holds {\em uniformly}. 

We will discuss two generalizations of the recursion theorem. 
First we discuss the ADN theorem, proved by Visser in 1980, with motivations
from $\lambda$-calculus, arithmetic provability, and the theory of numerations. 
We give a proof of the ADN theorem from the recursion theorem with parameters. 
Second, we discuss Arslanov's completeness criterion from 1977/1981, which 
generalizes the recursion theorem from computable functions to 
$A$-computable functions, for any set $A$ of incomplete c.e.\ degree. 
We then proceed to prove a joint generalization of these two theorems. 

Our notation from computability theory is mostly standard. 
Partial computable (p.c.) functions are denoted by lower case 
Greek letters, and (total) computable functions by lower case 
Roman letters. 
$\omega$ denotes the natural numbers, 
$\vph_e$ denotes the $e$-th p.c.\ function, 
and $W_e$ denotes the domain of $\vph_e$. 
We write $\vph_e(n)\darrow$ if this computation is defined, 
and $\vph_e(n)\uarrow$ otherwise. 
We let $\la e,n\ra$ denote a computable pairing function. 
$\emptyset'$ denotes the halting set. 
For unexplained notions we refer to Odifreddi~\cite{Odifreddi} or 
Soare~\cite{Soare}.

\section{Diagonally noncomputable and fixed point free functions}

A function $f$ is called {\em fixed point free}, or simply FPF, 
if $W_{f(n)}\neq W_n$ for every~$n$.
Note that by the recursion theorem no FPF function is computable. 
We will also consider {\em partial\/} functions without fixed points. 
Extending the above definition, we call a partial function $\delta$ 
FPF if for every $n$, 
\begin{equation}
\delta(n)\darrow \; \Longrightarrow \; W_{\delta(n)}\neq W_n. 
\end{equation}

A function $g$ is called {\em diagonally noncomputable\/}, or DNC, 
if $g(e)\neq \vph_e(e)$ for every~$e$.

We are interested in the Turing degrees of sets computing FPF or DNC functions.
It is well-known that these coincide:

\begin{proposition} {\rm (Jockusch et al.\ \cite{Jockuschetal})} \label{equivalence}
The following are equivalent for any set $A$:
\begin{enumerate}[\rm (i)]

\item $A$ computes a $\FPF$ function, 

\item $A$ computes a $\DNC$ function, 

\item $A$ computes a function $f$ such that 
$\vph_{f(e)} \neq \vph_e$ for every~$e$.

\end{enumerate}

\end{proposition} 

A proof of Proposition~\ref{equivalence} can be found in 
Downey and Hirschfeldt~\cite[p87]{DH}.
Degrees of $\{0,1\}$-valued DNC functions are also called PA-degrees, 
as they coincide with degrees of complete extensions of Peano arithmetic, 
cf.\ Jockusch and Soare~\cite{JockuschSoare1972a}, \cite[p89]{DH}.

The paper Kjos-Hanssen, Merkle, and Stephan \cite{KHMS} 
contains a discussion linking various classes of DNC functions 
to Kolmogorov complexity.
Since by Theorem~\ref{Arslanov} below (in combination with 
Proposition~\ref{equivalence}), the recursion theorem fails 
for a c.e.\ Turing degree if and only if it contains a DNC function, 
this also provides a link between the recursion theorem 
and Kolmogorov complexity.

\section{The ADN theorem} \label{sec:ADN}

Motivated by the $\lambda$-calculus, arithmetic provability, and 
the theory of numerations, 
Visser~\cite{Visser} proved the following generalization of the 
recursion theorem, called the ADN theorem. 
ADN theorem stands for ``anti diagonal normalization theorem''. 
Applications of this theorem are discussed in 
Visser (for the $\lambda$-cal\-culus and numerations) 
and Barendregt~\cite{Barendregt}, 
(in which the ADN theorem is used to prove a result of Statman's).
See also the discussion on computably enumerable 
equivalence relations and arithmetic provability in 
Bernardi and Sorbi~\cite{BernardiSorbi} and  
Montagna and Sorbi~\cite{MontagnaSorbi}.

It should be noted that the ADN theorem was originally formulated 
by Visser in a more general form, namely for arbitrary precomplete 
numberings, in the sense of Ershov~\cite{Ershov}. 
(Ershov also formulated a more general version of Theorem~\ref{recthm}
in this context.) We refrain from discussing this 
more general form here, but simply note that the version below
corresponds to the case of a standard numbering of the partial 
computable functions.

\begin{theorem} {\rm (ADN theorem, Visser~\cite{Visser})} \label{ADN}
Suppose that $\delta$ is a partial computable fixed point free
function. Then for every partial computable function $\psi$ there
exists a computable function $f$ such that for every $n$,
\begin{align}
\psi(n)\darrow \; &\Longrightarrow \; W_{f(n)}= W_{\psi(n)} \label{totalize} \\
\psi(n)\uarrow \; &\Longrightarrow \; \delta(f(n))\uarrow   \label{avoid}
\end{align}
\end{theorem}

\begin{definition}
In the following, if \eqref{totalize} holds for every~$n$, we will say 
that {\em $f$ totalizes $\psi$\/}. If both \eqref{totalize} and 
\eqref{avoid} hold, we will say that 
{\em $f$ totalizes $\psi$ avoiding~$\delta$}. 
\end{definition}

Note that the ADN theorem implies the recursion theorem: 
The function $\delta$ cannot be total, for otherwise $f(n)$ could 
not exist when $\psi(n)\uarrow$. 
It follows that there can be no computable FPF function. 
By Proposition~\ref{equivalence} this is equivalent to the 
statement of the recursion theorem.

Visser's original proof of the ADN theorem mimicked the proof of the 
recursion theorem. Here we prove the result directly from the 
recursion theorem with parameters, resulting in the following
short proof. 

\medskip\noindent
{\em Proof of Theorem~\ref{ADN}.}
By the recursion theorem with parameters (Theorem~\ref{recthmparam}), 
we can define a computable function $f$ such that 
$$
\vph_{f(n)}=
\begin{cases}
\vph_{\delta(f(n))} & \text{if $\delta(f(n))\darrow$ before $\psi(n)\darrow$}, \\
\vph_{\psi(n)} & \text{if $\psi(n)\darrow$ before $\delta(f(n))\darrow$}, \\
\uparrow & \text{otherwise.}
\end{cases}
$$
Now the first case can never obtain, since $\delta$ is FPF. 
Hence, if $\psi(n)\darrow$ then $W_{f(n)} = W_{\psi(n)}$, 
and if $\psi(n)\uarrow$ then $\delta(f(n))\uarrow$, 
so $f$ totalizes $\psi$ avoiding $\delta$.
\qed

\medskip
We note that by taking $\psi$ in Theorem~\ref{ADN} universal, 
we see that the following uniform version holds:

\begin{theorem} {\rm (ADN theorem, uniform version)} \label{ADNuniform}
Suppose that $\delta$ is a partial computable fixed point free
function. Then there exists a computable function $f$
such that for every $e$ the function 
$f(\la e,n\ra)$ totalizes $\vph_e$ avoiding $\delta$. 
\end{theorem}
\begin{proof}
Suppose that $\psi(\la e,n\ra) = \vph_e(n)$ is universal. 
By Theorem~\ref{ADN}, there exists a computable $f$ that 
totalizes $\psi$ avoiding $\delta$. 
But then, for fixed $e$, $f(\la e,n\ra)$ totalizes $\vph_e$ 
avoiding $\delta$:
$$
W_{f(\la e,n\ra)} = W_{\psi(\la e,n\ra)} = W_{\vph_e(n)}
$$
whenever $\vph_e(n)\darrow$, and 
$\delta(f(\la e,n\ra))\uarrow$ if $\vph_e(n)\uarrow$.
\end{proof}

In other words, if we can totalize a universal $\psi$ avoiding $\delta$, 
we can uniformly do the same for all $\vph_e$.

\section{Arslanov's completeness criterion}

Let $\emptyset'$ denote the halting set. Obviously, $\emptyset'$ can 
compute a FPF function. By the low basis theorem \cite{JockuschSoare1972b}, 
there exist FPF functions of low degree. However, these cannot have 
c.e.\ degree. 
On the other hand, by Ku\v{c}era~\cite{Kucera}, any FPF degree below 
$\emptyset'$ bounds a noncomputable c.e.\ degree. 
The Arslanov completeness criterion says that for 
c.e.\ degrees, containing a FPF function characterizes Turing 
completeness. Namely, if $A$ is a Turing incomplete c.e.\ set, then 
$A$ does not compute a FPF function. 
This shows that the recursion theorem holds for any function~$f$ that 
is computable by some incomplete c.e.\ set $A$, rather than just the 
computable ones. Arslanov proved this 
elegant characterization by building on work of Martin and Lachlan, 
cf.\ Soare~\cite[p88]{Soare}.

\begin{theorem} {\rm (Arslanov completeness criterion 
\cite{Arslanov}, \cite{ANS})} 
\label{Arslanov}
Suppose $A$ is c.e.\ and $A <_T \emptyset'$. 
If $f$ is an $A$-computable function, then $f$ has a fixed point, 
i.e.\ there exists $e\in\omega$ such that $W_{f(e)} = W_e$.
\end{theorem}

This result also holds for other notions of reduction, 
such as m- and wtt-reducibility (cf.\ \cite[p308,338]{Odifreddi}). 

Note the following analogous results for the PA-degrees mentioned above:
Again, by the low basis theorem, PA-degrees can be low, but by 
Jockusch and Soare~\cite{JockuschSoare1972a}, a c.e.\ PA-degree must 
be complete. 

The Arslanov completeness criterion has been extended in various ways, 
by considering relaxations of the type of fixed point, e.g.\ 
instead of requiring $W_{f(e)} = W_e$, 
requiring merely that 
$W_{f(e)}$ is a finite variant of~$W_e$ (Arslanov), 
or that they are Turing equivalent (Arslanov), 
or that the $n$-th jumps of these sets are Turing equivalent (Jockusch). 
In this way the completeness criterion can be extended to all levels 
of the arithmetical hierarchy.
For a discussion of these results we refer the reader to 
Soare \cite[p270 ff]{Soare} and 
Jockusch, Lerman, Soare, and Solovay~\cite{Jockuschetal}.
The latter paper also contains an extension of Theorem~\ref{Arslanov} 
from c.e.\ degrees to d.c.e.\ degrees. 

By the following proposition, for a given FPF function $\delta$,
we may assume without loss of generality that $\delta$ is
defined on the set $\bigset{a \mid W_a\neq \emptyset}$.
In particular, we see that the difficulty of producing a total 
FPF function lies in the set of c.e.\ codes for the empty set.
This will also play a role in the proof of Theorem~\ref{joint}
below.

\begin{proposition} \label{domain}
Suppose $\delta$ is a partial $\FPF$ function.
Then $\delta$ computes a partial $\FPF$ function $\hat\delta$ 
such that $\bigset{a \mid W_a\neq \emptyset}\subseteq \dom(\hat\delta)$.
\end{proposition}
\begin{proof}
Fix a code $d$ such that $W_d = \emptyset$, and define
\[
\hat\delta(a) =
\begin{cases}
\delta(a) & \text{if $\delta(a)\darrow$ before $W_a\neq\emptyset$}, \\
d    & \text{if $W_a\neq\emptyset$ before $\delta(a)\darrow$}, \\
\uparrow & \text{otherwise.}  
\end{cases}  
\] 
Cleary, $\hat\delta$ is FPF, and defined on 
$\bigset{a \mid W_a\neq \emptyset}$.
\end{proof}

\section{A joint generalization}

In this section we prove a joint generalization of the ADN theorem 
and Arslanov's completeness criterion. The latter generalizes the 
recursion theorem from computable functions to 
arbitrary functions bounded by an incomplete c.e.\ Turing degree. 
The ADN theorem is a statement about partial computable FPF 
functions~$\delta$. We show that this can also be generalized 
to partial $A$-computable FPF functions $\delta$, 
with $A$ c.e.\ and Turing incomplete. 
This results in a statement simultaneously generalizing  
Arslanov's completeness criterion and the ADN theorem. 

\begin{theorem} {\rm (Joint generalization)} \label{joint}
Suppose $A$ is a c.e.\ set such that $A <_T \emptyset'$. 
Suppose that $\delta$ is a partial $A$-computable fixed point free 
function. Then for every partial computable function $\psi$ there 
exists a computable function $f$ totalizing $\psi$ avoiding $\delta$, 
i.e.\ such that for every~$n$, 
\begin{align}
\psi(n)\darrow \; &\Longrightarrow \; W_{f(n)}= W_{\psi(n)} \label{totalize2} \\
\psi(n)\uarrow \; &\Longrightarrow \; \delta(f(n))\uarrow   \label{avoid2}
\end{align}
\end{theorem}
\begin{proof}
Suppose $A$ is a c.e.\ set, and suppose that $\delta$ is a 
partial $A$-computable FPF function for which the statement of 
the theorem does not hold, i.e.\ there exists a $\psi$ such that 
all $f$ fail to totalize $\psi$ avoiding~$\delta$. 
Without loss of generality, $\psi$ is universal by the same 
reasoning as in Theorem~\ref{ADNuniform}.
We prove that $\emptyset' \leq_T A$.

Given $x$, we define a computable function $f$ in such a way that the code 
of $f$ depends uniformly on~$x$. Later we will decide with $A$ whether 
$x\in \emptyset'$ depending on what happens with~$f$. 

We say that $f$ totalizes $\psi$ at stage $s$ if for all $n\leq s$
\begin{align*}
\psi_s(n)\darrow \; &\Longrightarrow \; W_{f(n),s}= W_{\psi(n),s} \\
\psi_s(n)\uarrow \; &\Longrightarrow \; W_{f(n),s}= \emptyset.
\end{align*}

Suppose that $\delta = \{e\}^A$, and consider the approximation
$$
\delta_s(x) = \{e\}^{A_s}_s(x).
$$
Note that if $\delta_s(x)\darrow$ the computation does not have to be 
permanent; it can change later if $A$ changes below the use, i.e.\ if 
$A_s\restr u \neq A\restr u$ for the use $u$ of this computation.
Also, it is possible that $\delta_s(x)\darrow$ infinitely often, but 
$\delta(x)\uarrow$.
Now if $\delta_s(x)\darrow$, and $A_s\restr u = A\restr u$,
then we know that the computation is permanent, 
i.e.\ $\delta_s(x) = \delta(x)$.
Because $A$ is c.e., we can see with $A$ whether a computation
$\delta(x)$ is permanent by checking that $A$ is permanent below the use.
So if we know that $\delta(x)\darrow$ for certain $x$, there have to
be stages $s$ where this computation is permanent, and hence we can find
such $x$ computably in $A$. 

Below, we will also search without $A$ for $x$ and $s$ such that 
$\delta_s(x)\darrow$ appears to be permanent. 
A complication here is that if $x$ is the least 
such number, there may be $y<x$ such that $\delta_t(y)\darrow$ at 
every stage $t$, but $\delta(y)\uarrow$, so our search never 
reaches~$x$.\footnote{Even if we use the convention, customary in 
infinite injury arguments, that $\delta_t(y)\uarrow$ for at least one 
stage $t$ after the computation changes, there might be several 
$y<x$ taking turns in preventing us to reach~$x$.}
To prevent this, we use a special order of search $L$
in which every number occurs infinitely often. 

{\em Construction of~$f$.} 
Given $x$, we define $f$ as follows. 
To make $f$ total, at stage $s=0$ we define 
$W_{f(n),0}=\emptyset$ for all~$n$. 

Fix a computable order $L$ of the natural numbers in which every 
number occurs infinitely often, for example
$$
1, 1, 2, 1, 2, 3, 1, 2, 3, 4, \ldots, 1, 2, 3, \ldots n, \ldots
$$
At stage $s$, we search for the $L$-least $n$ with $n\leq s$ 
(in the usual order) such that 
\begin{equation} \label{lock}
\psi_s(n)\uarrow \text{ and } \delta_s(f(n))\darrow,  
\end{equation}
and such that this copy of $n$ in $L$ was not previously discarded. 
If no such $n$ is found, we let $f$ totalize $\psi$ at stage~$s$.

If $n$ is found, we {\em freeze\/} $f$, meaning that for subsequent stages $t$, 
as long as $f$ is frozen we have $W_{f(n),s} = W_{f(n),t}$ for every~$n$.

We freeze $f$ until (if ever) we find a stage $t>s$ such that one of the 
following happens:

% N.B. following use of enumerate necessary since itemize has a very
% big indentation in amsart, different from enumerate 
\begin{enumerate}[\rm (i)]

\item[$\bullet$] 
At stage $t$ we see that $\delta_s(f(n))$ was not permanent, 
either because $\delta_t(f(n))\uarrow$ or because $A$ has changed 
below the use of the computation. 
Then we unfreeze $f$ by letting $f$ totalize $\psi$ at stage~$t$, 
we discard this copy of $n$ in $L$, 
and search for a new stage $s$ and $n$ as in~\eqref{lock}.

\item[$\bullet$] 
$x\in \emptyset_t'$.
Since $W_{f(n),s} = \emptyset$ (because $\psi_s(n)\uarrow$),
we can let $f(n)$ follow $\delta_s(f(n))$, meaning that
$W_{f(n)} = W_{\delta_s(f(n))}$.
(Note that this makes the definition of $f(n)$ circular, so formally
we are applying the recursion theorem with parameters
(Theorem~\ref{recthmparam}) here.)

\end{enumerate}
This concludes the construction of~$f$. We make the following observations. 

To fail the statement of the theorem, $\delta$ has to kill all 
functions $f$ totalizing $\psi$, (i.e.\ such that \eqref{totalize2} holds), 
by making \eqref{avoid2} fail. 
This means that $\delta$ has to {\em engage\/} some $f(n)$
with $\psi(n)\uarrow$ by becoming defined on it. 
Now as long as this has not happened, $f$ is totalizing $\psi$, so 
sooner or later $\delta(f(n))\darrow$ has to happen for some $n$, 
meaning that $\delta_s(f(n))\darrow$ is permanent from some stage 
$s$ onwards. 
Note that there may be fake computations $\delta_s(f(n))\darrow$, 
but every time such a computation is found to be fake, 
$f$ proceeds to totalize $\psi$ for at least one stage, thus 
keeping the pressure on~$\delta$.

Now if $x\notin\emptyset'$, then $f$ settles on such a permanent 
computation $\delta_s(f(n))$, meaning that at some stage $s$, 
$n$ as in \eqref{lock} is found, $f$ is frozen, and never changes 
after that. 
This is because of our special search procedure using the order $L$: 
If there are fake computations $\delta_t(y)\darrow$ with $y<f(n)$, 
then these are all discarded after finitely many stages. 
It may be that $y<f(n)$ and $\delta_t(y)\darrow$ for infinitely many~$t$, 
but every time $f$ freezes on such a computation, and this computation
changes later, a copy of $y$ in $L$ is discarded, and only finitely many 
copies of $y$ occur before $f(n)$ in~$L$.

{\em The decision procedure for $\emptyset' \leq_T A$.}
Given $x$, we decide whether $x\in\emptyset'$ using $A$. 
Note that the definition of $f$ above uniformly depends on~$x$
(meaning that there is a computable function $h$ such that 
$f = \vph_{h(x)}$ for every~$x$).  
Following the definition of $f$, we search for a point $f(n)$
where $\delta$ engages $f$ by becoming defined. 
As pointed out above, such a point exists, since otherwise 
$f$ would totalize $\psi$, and also \eqref{avoid2} would hold. 
Hence we find a stage $s$ and $n$ such that 
$\delta_s(n)\darrow$ and $f$ is frozen at stage~$s$. 
If the computation $\delta_s(n)\darrow$ is not permanent
(which we can see by checking $A_s\restr u \neq A\restr u$, 
with $u$ the use of the computation), then we search on. 
Continuing in this way, we either find a stage~$s$ 
with $x\in \emptyset'_s$, or a pair $n$ and $s$ such that 
the computation $\delta_s(f(n)) = \delta(f(n))\darrow$ is permanent. 

In the first case we know of course that $x\in\emptyset'$, 
and we are done. 
Note that it could happen that $x$ enters $\emptyset'$ at 
a stage $t$ where $f$ is frozen on a nonpermanent computation 
$\delta_s(f(n))$, and hence $f(n)$ starts to follow 
this nonpermanent computation $\delta_s(f(n))$. In that 
case there does not need to be any final permanent computation 
$\delta(f(n))$. 
But this is immaterial for our case distinction: We either find 
that $x\in \emptyset'$ or we find a permanent computation $\delta_s(f(n))$.

In the second case we know that $f$ is eventually frozen on the 
permanent computation $\delta_s(f(n))$, and that it will never 
change to another~$n$. In this case we also know that 
$x\notin\emptyset'$, for otherwise $f(n)$ would start to follow 
$\delta_s(f(n))$, and hence $\delta_s(f(n))$ would have to change. 
Thus we can decide $x\in\emptyset'$ for every~$x$.
\end{proof}

Note that for computable $A$, the statement of Theorem~\ref{joint}
is the same as Visser's ADN theorem. 
Also, the theorem implies Arslanov's completeness criterion, 
since $f$ as in the theorem cannot exist if $\delta$ is total, so that 
in particular any total $A$-computable $\delta$ cannot be FPF, 
hence must have a fixed point. 

As we have seen, the recursion theorem is effective, in the sense
that the fixed point from the statement of the theorem can be 
found effectively. This is the content of the recursion theorem 
with parameters (Theorem~\ref{recthmparam}).
One may wonder if the same holds for the ADN theorem 
(when appropriately formulated), Arslanov's completeness criterion, 
or their joint generalization Theorem~\ref{joint}.
Note that the ADN theorem can be proved using the recursion theorem with
parameters (see the proof given in Section~\ref{sec:ADN} above). 
It stands to reason that its generalized version might be provable from
a parameterized version of Arslanov's completeness criterion.
However, this does not work: As it turns out, neither the ADN theorem nor 
Arslanov's completeness criterion has a version with parameters.
This is proven in \cite{Terwijn}. 
As there are no parameterized versions of these theorems, a fortiori, 
neither is there such a version for their joint generalization. 

Finally, we remark that the ADN theorem is effective in the following sense. 
We have already seen that it is uniform in a code of $\psi$ 
(as observed in Theorem~\ref{ADNuniform}).
Also, from the proof of the ADN theorem in section~\ref{sec:ADN},
it is clear that the code of the function $f$ depends
effectively on a code of~$\delta$.
Hence the result is uniform in both $\psi$ and~$\delta$.
A similar uniformity does not hold for the joint generalization
Theorem~\ref{joint}: It is still uniform in $\psi$, but not in 
codes for $A$ and $\delta$, as is proven in \cite{Terwijn}. 
Hence the proof of Theorem~\ref{joint} given above is 
necessarily nonuniform.

\end{document}